\documentclass[a4paper,12pt]{article}

\usepackage{amssymb, amsmath, latexsym, amsthm, cite}

\newcommand{\Q}{\mathbb{Q}}

\newcommand{\F}{\mathbb{F}}

\newcommand{\R}{\mathbb{R}}

\newcommand{\C}{\mathbb{C}}

\renewcommand{\b}[1]{{\bf #1}}

\newcommand{\x}{\b{x}}

\newcommand{\Z}{\mathbb{Z}}

\newtheorem{theorem}{Theorem}

\newtheorem{corollary}{Corollary}

\newtheorem{lemma}{Lemma}

\newtheorem{question}{Question}

\newtheorem{result}{Result}

\author{Jahan Zahid\footnote{This research is supported by EPSRC: Grant wef 1 Oct 2006: EP/P502667/1}  \\ Mathematical Institute\\ University of Oxford}

\date{}

\title{Zeros of $p$-adic forms}

\begin{document}

\maketitle

\section{Introduction}
\subsection{Background and statement of main result}

We shall be interested in the existence of non-trivial $p$-adic zeros of $p$-adic forms. Recall that a form is a homogeneous polynomial in several variables. From now on we shall say a zero to mean a non-trivial zero unless we state otherwise. To set the scene, Artin originally thought about $C_{i}$ fields. A $C_{i}$ field is one in which every form of degree $d$ in $n$ variables is guaranteed to have a zero provided that $n > d^{i}$. For example $\C$ is a $C_{0}$ field, every finite field $\F_{q}$  is a $C_{1}$ field (this is the Chevalley-Warning theorem) and $\F_{q}((t))$ is a $C_{2}$ field (this was proved by Lang \cite{Lan52}). So in some sense one can view this $C_{i}$ property as a measure of how `algebraically closed' a field is. Artin then conjectured \cite{Art65} that $\Q_{p}$ is a $C_{2}$ field. We should note that one can easily construct forms with $n = d^{2}$ which only have the trivial $p$-adic zero so this is in fact the best one can hope for. It had already been well known that for the case of quadratic forms this conjecture holds. Demyanov \cite{Dem50} and Lewis \cite{Lew52} independently showed that the conjecture holds for cubic forms (Demyanov requiring that $p \ne 3$). Artin's conjecture was proved to be false by Terjanian \cite{Ter66} who exhibited a quartic form with 18 variables which only has the trivial zero in $\Q_{2}$. It is however still open whether the result holds for odd degree forms.

The first breakthrough in the general question came when Brauer \cite{Bra45} proved that for all $d$ there exists a least integer $v_{d}$ such that all forms of degree $d$ with $n > v_{d}$ have a $p$-adic zero for every prime $p$. Much work has been spent on trying to get the best possible bounds for $v_{d}$, the current record being due to Wooley \cite{Woo98} who has proved the neat result that 
\begin{equation} \label{woolbd}
 v_{d} \leq d^{2^{d}}. 
\end{equation}
Wooley's analysis does give us better bounds on $v_{d}$ than (\ref{woolbd}), which we shall describe in due course. We should establish the notation $v_{d}(p)$ at this point, to mean the least integer such that every form of degree $d$ with $n > v_{d}(p)$ has a non-trivial $p$-adic zero. 

In another direction Ax and Kochen \cite{AxKoc65} proved using methods from Mathematical Logic that for any $d$ there exists a least positive integer $p(d)$ such that provided $p \ge p(d)$ then every form with $n > d^{2}$ has a $p$-adic zero. So in some sense Artin's conjecture is `almost' true. Later work by Brown \cite{Bro78} provided a bound for $p(d)$. If we write $a \uparrow b$ for $a^{b}$ and $a \uparrow b \uparrow c$ for $a \uparrow (b \uparrow c)$  then Brown showed that

\[ p(d) \le 2 \uparrow 2 \uparrow 2 \uparrow 2 \uparrow 2 \uparrow d \uparrow 11 \uparrow (4d). \] 

As we have established a fair amount of notation at this point it is worth noting a few results that are currently known about the values of $v_{d}$ and $p(d)$. So for example Lewis' result tells us that $v_{3} = 9$ and hence $p(3) = 1$, Demyanov's counterexample tells us that $v_{4} \ge 18$ and $p(4) \ge 2$. Recently Heath-Brown \cite{Hea08} showed that $v_{4} \le 9144$, with $v_{4}(p) \le 312$ for $p \ne 2$. Leep and Yeomans \cite{LeeYeo96} showed that $p(5) \le 47$ and this was subsequently improved by Heath-Brown \cite{Hea08} to $p(5) \le 17$. Hence $v_{5}(p) = 25$ for $p \ge 17$.
\begin{question}
What is the best bound we can we get for $v_{5}$?
\end{question} 
Wooley's bound (\ref{woolbd}) tells us that
\begin{equation} \label{32}
v_{5} \leq 5^{32} \approx 2.3 \times 10^{22}.
\end{equation} 
But in fact we know already that $v_{5}(p) = 25$ for $p \ge 17$, so by considering the primes $2 \le p \le 13$ Wooley's analysis gives us more information than this viz.
\begin{equation} \label{wool5}
v_{5} \leq 6792217044067.
\end{equation} 
This gives a considerable improvement to (\ref{32}) and it is with this better estimate our new bounds shall be compared. We should now note the main results of this paper.

\begin{theorem} \label{main1}
We have 

\begin{enumerate} 

\item[(a)]
\begin{enumerate}
\item[(i)] $v_{5}(2) \le 300509,$
\item[(ii)] $v_{5}(3) \le 520329,$
\item[(iii)] $v_{5}(7) \le 180164,$
\item[(iv)] $v_{5}(13) \le 179592;$
\end{enumerate}

\item[(b)] $v_{5}(11) \le 348497.$

\end{enumerate}
\end{theorem} 

\begin{theorem} \label{main2}
We have $v_{5}(5) \le 4562911.$
\end{theorem}

The result for $p=5$ has been stated separately since a non-standard version of Hensel's lemma is required to lift a zero modulo $25$ to a zero in $\Q_{5}$. There are also some differences for the case $p=11$. However the same version of Hensel's lemma is used throughout in proving Theorem \ref{main1}. 
The difference in approach between our argument and Wooley's, is in the former requiring the vanishing of a `smaller' system of auxiliary forms than the latter, resulting in less stringent conditions on the number of variables to achieve this. 

We can also prove some results concerning systems of cubic and quadratic forms. However to be able to describe them effectively we need to introduce some notation first. We assume we are working over the field $\Q_{p}$. Given any system of $r_{1},r_{2},\ldots,r_{d}$ forms of degrees $1,2,\ldots,d$ respectively in $n$ variables, we define $V(r_{d},\ldots,r_{1})$ to be the least integer $V$ such that the system above has a zero, provided that $n>V$. More generally we shall write,  $V^{(m)}(r_{d},\ldots,r_{1})$ to mean the least integer $V$ such that any system as above vanishes on some subspace (over $\Q_{p}$) of dimension strictly greater than $m$, provided that $n>V$. For convenience we set $V^{(m)}(0,0,\ldots,0)=m$. Also we may sometimes write $V(r_{d},\ldots,r_{1};p)$ to indicate that we are restricting to the field $\Q_{p}$ for some fixed prime $p$. We should note that the following relation is fairly obvious.
\begin{eqnarray} 
V^{(m)}(r_{d},\ldots,r_{1}) & = & V^{(m)}(r_{d},\ldots,r_{2},0) + r_{1}. \label{lineq}
\end{eqnarray}

Although Artin's conjecture was originally stated for single forms of a certain degree we could have also stated it in the following way.

\begin{equation} \label{genconj}
V(r_{d}, \ldots, r_{2}, r_{1}) = r_{1} + 4r_{2} + \ldots + d^{2}r_{d}.
\end{equation}

So we can ask questions concerning systems of forms rather than a single form, and as it will turn out our analysis will require us to think more generally about systems rather than single forms. In fact it was Lang \cite{Lan52} (at the time a student of Artin) who proved that if $\Q_{p}$ is $C_{2}$ then this implies (\ref{genconj}). But as Terjanian showed that $\Q_{p}$ is not $C_{2}$, it is still worth studying systems for their own sake. The first result concerning systems of forms came from Birch, Lewis and Murphy \cite{BirLewMur62} who proved that

\[ V(2,0) = 8. \]
So Artin's conjecture holds for a system of two quadratic forms. The next two simplest cases after this concerns a system of a cubic and a quadratic form and a pair of cubic forms. However it is still unknown whether Artin's conjecture holds in these cases viz. do we have

\[ V(1,1,0) = 13 \qquad \mbox{and} \qquad V(2,0,0) =18? \]
We shall give bounds for these quantities by using the following theorem.

\begin{theorem} \label{newresult}
Let $r_{1},r_{2},r_{3}$ be non-negative integers, with $r_{3}>0$. For $p \ne 3$ we have 
\[ V(r_{3},r_{2},r_{1};p) \le V(r_{3}-1,r_{2}+6(r_{3}-1),r_{1}+6r_{2}+9r_{3};p). \]
\end{theorem}

We shall give a proof of this in due course. However we should note the immediate corollary

\begin{corollary} 
For $p \ne 3$ we have
\begin{enumerate}
\item[(i)] $V(1,1,0;p) \le 19$, 
\item[(ii)] $V(2,0,0;p) \le 119$.
\end{enumerate}
\end{corollary}

\begin{proof}
By Theorem \ref{newresult} and (\ref{lineq}) we have
\[ V(1,1,0;p) \le V(1,15;p) = 4 + 15 = 19, \]
and
\[ V(2,0,0;p) \le V(1,6,18;p) \le V(6,0;p) + 63. \] 
We shall see later (Lemma \ref{sysquad}) that, $V(6,0;p) \le 56$. Hence we obtain
\[ V(2,0,0;p) \le 119. \]
\end{proof}

To put these results into context we should note that the best bound one can obtain for $V(1,1,0)$ is by using the techniques of Leep and Schmidt \cite{LeeSch83}, where one can show that 

\[ V(1,1,0) \le 22. \]
Also Dietmann and Wooley \cite{DieWoo03} showed 

\[ V(2,0,0) \le 298. \]
They also show that 

\[ V(2,0,0;3) \le 233. \]
It might be possible to extend Theorem \ref{newresult} to the case $p=3$. This would require us to show  form has a zero modulo $9$ which lifts to a zero in $\Q_{3}$.

We shall now briefly focus our attention to forms over $\Q$. We can ask the same questions that we asked about forms over the local field $\Q_{p}$: is there an equivalent to Brauer's theorem \cite{Bra45} for $\Q$? The answer is no: there always exist positive definite forms which don't even have a zero in $\R$. For example 
\[ F(x_{1},\ldots,x_{n}) = x_{1}^{2k}+\ldots+x_{n}^{2k} \]
has only the trivial zero for all $k \ge 1$. However it is a remarkable theorem of Birch \cite{Bir57} that for all odd $d$ there exists a least integer $w_{d}$ such that all forms of degree $d$ over $\Q$ with $n > w_{d}$ have a zero. Birch famously said that his bounds for $w_{d}$ are ``not even astronomical''. Heath-Brown \cite{Hea07} has shown that $w_{3} \le 14$, by using the circle method approach. More recently Wooley \cite{Woo07} has shown that $w_{5} \le 1.38 \times 10^{14}$, by employing an elegant idea which involves finding zeros of the form in an algebraic extension of $\Q$.

If we restrict ourselves to non-singular forms, much better bounds can be obtained for the required number of variables to guarantee solubility. If we denote $w_{d}^{*}$ to mean $w_{d}$, where we have restricted to non-singular forms, then Birch proved \cite{Bir61} that $w_{d}^{*} \le (d-1)2^{d}$ provided that we have $p$-adic solubility for all $p$. Heath-Brown \cite{Hea84} has shown that $w_{3}^{*} \le 9$. Moreover Theorem's \ref{main1} and \ref{main2} imply that $w_{5}^{*} \le 4562911$.

\section{Quintic forms}

In this section we shall show how one obtains Wooley's bound (\ref{wool5}) and go on to prove Theorems \ref{main1} and \ref{main2}. 

\subsection{Some Preliminaries}

We should at this point establish the notation $\phi_{d}(p)$ to mean the least integer such that every diagonal form of degree $d$ with $n>\phi_{d}(p)$ variables has a non-trivial $p$-adic zero, and similarly $\phi_{d} =$ sup$_{p}\{\phi_{d}(p)\}$. Clearly $\phi_{d} \le v_{d}$. Moreover it was shown by Davenport and Lewis \cite{DavLew63} that for every $d$ we have $\phi_{d} \le d^{2}$, thus establishing Artin's conjecture for the class of diagonal forms. The key result in Wooley's argument as found in \cite[Lemma 2.1]{Woo98} is as follows.

\begin{lemma}
Let $d$ and $r_{i}$ $(1 \le i \le d)$ be non-negative integers with $d \ge 2$ and $r_{d} > 0$. Then provided $\phi_{d}< \infty$ one has

\begin{equation}  \label{2.1}
V(r_{d},r_{d-1},\ldots,r_{1}) \le \phi_{d} + V(r'_{d},r'_{d-1},\ldots,r'_{1})
\end{equation} 
where $r'_{d} = r_{d}-1$ and

\[ r'_{j} = \sum_{i=j}^{d}r_{i}\binom{\phi_{d}+i-j-1}{i-j} \qquad (1 \le j < d). \]

\end{lemma}

Heath-Brown \cite{Hea08} has made a small improvement to this replacing (\ref{2.1}) by 

\begin{equation} \label{hea1}
V(r_{d},r_{d-1},\ldots,r_{1}) \le V(r'_{d},r'_{d-1},\ldots,r'_{1}).
\end{equation}

We begin by using (\ref{2.1}) to prove (\ref{wool5}). From (\ref{2.1}) we have

\begin{equation} \label{v5}
v_{5}(p) = V(1,0,0,0,0;p) \le
 \phi+ V\Big(\phi,\binom{\phi+1}{2},\binom{\phi+2}{3},\binom{\phi+3}{4};p \Big),
 \end{equation}
where we have set $\phi = \phi_{5}(p)$ for brevity. Furthermore in writing $\varphi = \phi_{4}(p)$ we have

\[ V(a,b,c,d;p) \le \varphi + V(a',b',c',d';p) \]
where 

\[ a' = a-1, \qquad b'= b + a\varphi, \qquad c' = c + b\varphi +a\frac{\varphi(\varphi+1)}{2}, \]
\[ d' = d + c\varphi +b\frac{\varphi(\varphi+1)}{2} + a\frac{\varphi(\varphi+1)(\varphi+2)}{6}.  \]
It follows from an easy induction argument that by applying (\ref{2.1}) $a$ times we have

\begin{equation} \label{quarticbd}
V(a,b,c,d;p) \le a\varphi + V(\alpha,\beta,\gamma;p)
\end{equation}
where 

\[ \alpha = b + \frac{\varphi}{2}a(a+1),  \]
\[ \beta = c+ \varphi ab + \frac{\varphi(\varphi+1)}{2}a^{2} - \frac{\varphi}{4}a(a-1) + \frac{\varphi^{2}}{12}a(a-1)(4a+1),   \] 
\[ \gamma = d + \frac{\varphi}{6}a(2a+3b+6c +3 \varphi (a+b) + \varphi^{2}a) + \frac{\varphi}{12}a(a-1)(-2+6 \varphi (a+b) + 3 \varphi^{2}a) \]
\[ + \frac{\varphi^{2}}{12}a(a-1)(2a-1)(\varphi a -1) - \frac{\varphi^{3}}{24}a^{2}(a-1)^{2}. \]
\\
Moreover by setting $\psi = \phi_{3}(p)$, we can show similarly by induction that by applying (\ref{2.1}) $a$ times, we have

\begin{equation} \label{cubicbd}
V(a,b,c;p) \le a \psi + V(\alpha,0;p) + \beta
\end{equation}
where

\[ \alpha = b + \frac{\psi}{2}a(a+1), \]
\[ \beta = c + \psi ab + \frac{\psi^{2}}{3} a(a^{2}-1) + \frac{\psi (\psi+1)}{2}\frac{a(a+1)}{2}. \]
\\
At this point we should say something about $V(r,0)$. We could iteratively apply (\ref{2.1}) again, however one can do slightly better than this. For ease of notation, let $u(r,m;p) = V^{(m)}(r,0;p)$ and we shall denote $u(r;p)=u(r,0;p)$. In proceeding we shall note two results from Heath-Brown's paper \cite[Lemma 1 and 2]{Hea08}.

\begin{lemma} \label{sysquad}
For every prime $p$ we have
\begin{enumerate}
\item[(i)]
$u(1;p)=4$,
\item[(ii)]
$u(2;p)=8$,
\item[(iii)]
$u(3;p) \le 16$,
\item[(iv)]
$u(4;p) \le 24$,
\item[(v)]
$u(5;p) \le 40$,
\item[(vi)]
$u(6;p) \le 56$,
\item[(vii)]
$u(r;p) \le 2r^{2}-16$ for even $r \ge 8$,
\item[(viii)]
$u(r;p) \le 2r^{2}-14$ for odd $r \ge 7$.
\end{enumerate}
\end{lemma} 

For $p \ge 11$ it was proved by Schuur \cite{Sch80} that $u(3;p) = 12$. It transpires that for $p \ge 11$ one can do slightly better.
\begin{lemma} \label{L5}
For every prime $p \ge 11$ we have
\begin{enumerate}
\item[(i)]
$u(3;p)=12$,
\item[(ii)]
$u(4;p) \le 24$,
\item[(iii)]
$u(5;p) \le 32$,
\item[(iv)]
$u(6;p) \le 56$,
\item[(v)]
$u(r;p) \le 2r^{2}-2r-12$ for $r \equiv 1$ (mod $3$) and $r \ge 7$,
\item[(vi)]
$u(r;p) \le 2r^{2}-2r-8$ for $r \equiv 2$ (mod $3$) and $r \ge 8$,
\item[(vii)]
$u(r;p) \le 2r^{2}-2r-8$ for $r \equiv 0$ (mod $3$) and $r \ge 9$.
\end{enumerate}
\end{lemma} 

In order to get a bound on $v_{5}$ all we need now is some information on $\phi_{d}(p)$ for $2 \le d \le 5$ and $2 \le p \le 13$. The techniques used to study $\phi_{d}(p)$ are fairly routine, see the work of Davenport and Lewis \cite{DavLew63} for example, so we shall state the following lemma without proof.

\begin{lemma} \label{table}
The following table contains the values $\phi_{d}(p)$ for, $2 \le d \le 5$ and $2 \le p \le 13$\\ 
\[ \begin{tabular}{c | c | c | c | c}
\hline
$p$ & $\phi_{2}(p)$ & $\phi_{3}(p)$ & $\phi_{4}(p)$ & $\phi_{5}(p)$ \\ 
\hline 
$2$ & $4$ & $3$ & $15$ & $5$ \\ 
$3$ &  $4$ & $4$ & $8$ & $5$ \\
$5$ & $4$ & $3$ & $16$ & $7$ \\
$7$ & $4$ & $6$ & $8$ & $5$ \\
$11$ & $4$ & $3$ & $8$ & $15$ \\
$13$ & $4$ & $6$ & $12$ & $5$ \\
\end{tabular} \]
\end{lemma} 

It transpires that the worst bound for (\ref{v5}) occurs whenever $p=11$ since this gives the largest value of $\phi_{5}(p)$. For $p=11$ we compute 
\[ v_{5}(11) \le 15 + V(15,120,680,3060). \]
Now by applying (\ref{quarticbd}) we obtain
\[ V(15,120,680,3060) \le 120 + V(1080,91080,4410900). \]
So now applying (\ref{cubicbd}) we get
\[ V(1080,91080,4410900) \le 3240 + u(1842300;11) + 4082145300. \]
Finally by Lemma \ref{L5} we have
\[ u(1842300;11) \le 6788134895392. \]
Hence we obtain (\ref{wool5}) viz.
\[ v_{5} \le 6792217044067. \]

We should note that by using Heath-Brown's improvement (\ref{hea1}) we may replace (\ref{cubicbd}) by
\begin{equation} \label{impcubicbd}
V(a,b,c;p) \le V(\alpha,0;p) + \beta
\end{equation}
where

\[ \alpha = b + \frac{\psi}{2}a(a+1), \]
\[ \beta = c + \psi ab + \frac{\psi^{2}}{3} a(a^{2}-1) + \frac{\psi (\psi+1)}{2}\frac{a(a+1)}{2}. \]

\subsection{The Proof of Theorems \ref{main1} and \ref{main2}}

We shall now develop the hybrid approach to Artin's problem as found in Heath-Brown \cite{Hea08}. We shall argue as Heath-Brown does for cubic and quartic forms. However in some cases we shall need a larger `admissible set'. Arguing by contradiction, we suppose that $F(\x) \in \Q_{p}[\x]$ is a form of degree 5 in $n$ variables with only the trivial $p$-adic zero. We will find conditions on $n$ which give us a contradiction and thereby obtain a bound for $v_{5}(p)$. Let $\theta_{s}(F(\x))$ denote the obvious reduction of $F(\x) \in \Z_{p}[\x]$ modulo $p^{s}$, for $s \ge 1$. Our strategy shall be the following, we will find a collection of vectors $\b{e}_{1}, \ldots , \b{e}_{k} \in  \Q_{p}^{n}-\{\b{0}\}$ such that (for some $r \in \Z$) by applying Hensel's Lemma to the form $\theta_{s}(p^{-r}F(t_{1}\b{e}_{1}+\ldots+t_{k}\b{e}_{k}))$, we shall be able to lift a suitably non-singular zero modulo $p^{s}$ to a non-trivial zero in $\Z_{p}$. It will follow that we will need to take $s=1$ and $s=2$ for $p \in \{2,3,7,11,13 \}$ and $p=5$ respectively, each case requiring a different version of Hensel's Lemma.

When $\x \in \Q_{p}^{n}-\{ \b{0} \}$ we say that $\x$ has ``level $r$'' where $0 \le r \le 4$, if $v(F(\x)) \equiv r$ (mod $5$). Since we have assumed that $F(\x) \ne 0$, the level of $\x$ is a well-defined property. Given any quadratic form 
\[ Q(x_{1},\ldots,x_{n}) = \sum_{i \le j}c_{ij}x_{i}x_{j} \]
we define the ``length of $Q$'' by
\[ l(Q) := \# \{ c_{ij} \ne 0 : 1 \le i \le j \le n \}. \]
Clearly $0 \le l(Q) \le \binom{n}{2}$. For any set 
\[ S = \{ \b{e}_{1}, \ldots , \b{e}_{m} \} \subset \Q_{p}^{n}-\{\b{0}\} \]
we shall say that $S$ is ``$(s,t)$-admissible'' provided that the following conditions hold
\begin{enumerate}
\item[(i)]
For each vector $\b{e_{i}} \in S$, we have $0 \le v(F(\b{e_{i}})) \le 4$,
\item[(ii)]
There is some subset $S_{r} \subseteq S$ of cardinality $s$, whose vectors all have level $r$,
\item[(iii)]
If $\b{e}_{i_{1}},\ldots,\b{e}_{i_{s}} \in S$ all have the same level, with $i_{1} < i_{2}<\ldots<i_{s}$ then, by writing $\b{e}_{i_{r}}=\b{e}_{r}'$ we have
\[ F(t_{1}\b{e}_{1}'+\ldots+t_{k}\b{e}_{s}') = \sum_{i}a_{i}t_{i}^{5}+\sum_{i < j}b_{ij}t_{i}t_{j}^{4}
+ \sum_{i<j<k}c_{ijk}t_{i}t_{j}t_{k}^{3} \]
for some $a_{i},b_{ij}, c_{ijk} \in \Q_{p}$. Furthermore we have the condition that for any $3 \le k \le n$
\[ l\Big( \sum_{i<j}c_{ijk}t_{i}t_{j} \Big) \le t. \]
\end{enumerate}
We shall show in due course that $(s,t)$-admissible sets exist for every $s \ge 2$ and $t \ge 0$, provided that $n$ is large enough. Before proceeding further we shall note that if $\b{e}_{i},\b{e}_{j}, \b{e}_{k} \in S_{r}$ then, $p^{-r}F(x\b{e}_{i}+y\b{e}_{j}+z\b{e}_{k})$ must have coefficients in $\Z_{p}$. This comes immediately from the following result.

\begin{lemma} \label{coeff}
Let 
\[ f(x,y,z) = ax^{d}+bxy^{d-1}+cy^{d} + dxyz^{d-2} + (ex+fy)z^{d-1} + gz^{d} \in \Q_{p}[x,y,z] \]
where $a,c,g \in \Z_{p}$. Suppose that $f(x,y,z)$ has only the trivial zero in $\Q_{p}$, then $b, d, e, f \in \Z_{p}$. 
\end{lemma}

\begin{proof}
We need to prove that $b,d,e,f$ are all $p$-adic integers. We shall argue by contradiction. Suppose that
\[ s:= min\{ v(b),v(d),v(e),v(f) \} < 0 \]
Then
\begin{equation} \label{thetaf}
\theta(p^{-s}f(x,y,z)) = b'xy^{d-1} +  d'xyz^{d-2} + (e'x+f'y)z^{d-1} \in \F_{p}[x,y,z] 
\end{equation}
where at least one of $b',d',e'$ or $f'$ is non-zero. If $b' \ne 0$, then $(0,1,0)$ is a non-singular solution to (\ref{thetaf}), which by Hensel's Lemma lifts to a non-trivial $p$-adic zero of $f(x,y,z)$. This contradicts our assumption that $f$ has only the trivial zero, therefore $b' = 0$. Similarly if either $e'$ or $f'$ $\ne 0$ then $(0,0,1)$, is a non-singular solution to (\ref{thetaf}). Hence as before we deduce that $e',f' = 0$. This forces $d' \ne 0$, however in this case $(0,1,1)$ is a non-singular solution to (\ref{thetaf}). So as before we deduce that $d'=0$. Hence $s \ge 0$, so $b,d,e,f \in \Z_{p}$ as required.
\end{proof}
To ensure linear independence we note the following result from Heath-Brown's paper \cite[Lemma 4]{Hea08}.

\begin{lemma} \label{linind}
Let $F(\x) \in \Q_{p}[x_{1},\ldots,x_{n}]$ be a form of degree $d$, having only the trivial zero in $\Q_{p}^{n}$. Let $\b{e}_{1},\ldots,\b{e}_{m}$ be linearly independent vectors in $\Q_{p}^{n}$, and suppose we have a non-zero vector $\b{e}$ such that the form 
\[ F_{0}(t_{1},\ldots,t_{m},t) := F(t_{1}\b{e}_{1}+\ldots+t_{m}\b{e}_{m}+t\b{e}) \]
contains no terms of degree one in $t$. Then the set $\{ \b{e}_{1},\ldots,\b{e}_{m},\b{e} \}$ is linearly independent.
\end{lemma}

Hence it follows that if we have a $(s,t)$-admissible set with a subset $S_{r}$ of cardinality $s$ as in (ii), then  $S_{r}$ is a linearly independent set. To prove this we note that without loss of generality
\[ S_{r} = \{ \b{e}_{1},\ldots, \b{e}_{s} \}. \]
Then by (iii) 
\[ F(t_{1}\b{e}_{1}+\ldots+t_{l}\b{e}_{l}) \]
contains no terms of degree one in $t_{l}$, for $l \le s$. So by applying Lemma \ref{linind} inductively for $l \le s$ we deduce that $\{ \b{e}_{1},\ldots,\b{e}_{l} \}$ is a linearly independent set. If we have a $(s,t)$-admissible set with a subset $S_{r}$, then it follows that
\[ p^{-r}F(t_{1}\b{e}_{1}+\ldots+t_{k}\b{e}_{s}) = \sum_{i}a_{i}t_{i}^{5}+\sum_{i < j}b_{ij}t_{i}t_{j}^{4}+\sum_{i<j<k}c_{ijk}t_{i}t_{j}t_{k}^{3} \in \Z_{p}[t_{1},\ldots,t_{s}] \]
where $v(a_{i})=0$, i.e. all the $a_{i}$ are $p$-adic units. This can be seen by applying Lemma \ref{coeff}, to the form $p^{-r}F(t_{i}\b{e}_{i}+t_{j}\b{e}_{j}+t_{k}\b{e}_{k})$, for each triple $(i,j,k)$. We shall now show that $(s,t)$-admissible sets exist provided that $n$ is large enough. We define $(u)_{+} := max\{u,0\}$, for any $u \in \R$.

\begin{lemma} \label{kad}
Let $F$ be a form of degree $5$ in $n$ variables, with only the trivial zero over $\Q_{p}$. Then for $s \ge 2$ and $t \ge 0$ there exists a $(s,t)$-admissible set $S$ provided that
\begin{equation} \label{k-ad}
n > V(r_{3},r_{2},r_{1};p)
\end{equation}
where
\[ r_{3} = 5\bigg( \binom{s}{2}-t \bigg)_{+}\;, \quad r_{2} = 5\binom{s+1}{3}, \quad r_{1} = 5\binom{s+2}{4} \]
\end{lemma}

\begin{proof}
We shall prove this by inductively constructing a $(s,t)$-admissible set. Let $\b{e}_{1} \in \Q_{p}^{n} - \{\b{0}\}$, then we can ensure that (i) holds by multiplying $\b{e}_{1}$ by an appropriate power of $p$. Next suppose that (i) and (iii) hold for the set $S= \{ \b{e}_{1},\ldots,\b{e}_{m} \}$, where $m \ge 1$. We need to find a vector $\b{e}_{m+1}$ such that (i) and (iii) hold for $S \cup \{\b{e}_{m+1}\}$. Suppose that in $S$ there are $l_{i}$ vectors of level $i$, where $0 \le l_{i} < s$ and $0 \le i \le 4$ (note that if $l_{i} \ge s$ then $S$ is already $(s,t)$-admissible since (ii) holds). Moreover let 
\[ L_{i} = \{ \b{e}_{1}^{(i)}, \ldots, \b{e}_{l_{i}}^{(i)} \} \subseteq S \]
be the subset of $S$ which contains all the vectors of level $i$. Clearly $S = L_{0} \cup \ldots \cup L_{4}$. Now consider
\[ F(t_{1}\b{e}_{1}^{(i)}+\ldots+t_{l_{i}}\b{e}_{l_{i}}^{(i)}+t\b{e}_{m+1}) =  \sum_{|\b{u}| \le 5} \b{t}^{\b{u}}t^{5-|\b{u}|}F_{i}(\b{e}_{m+1};\b{u}) \]
where 
\[ \b{t}^{\b{u}} = \prod_{j=1}^{l_{i}} t_{j}^{u_{j}}. \]
and $F_{i}(\b{e}_{m+1};\b{u})$ is a form of degree $5- |\b{u}|$ in $\b{e}_{m+1}$. Since \emph{a priori} we don't know what the level of $\b{e}_{m+1}$ will be, it follows that for $S \cup \{\b{e}_{m+1}\}$ to satisfy (iii), we require that for each $0 \le i \le 4$
\[ F_{i}(\b{e}_{m+1};\b{u})=0, \quad \mbox{for all} \; \; 3 \le |\b{u}| \le 4  \]
and
\[ F_{i}(\b{e}_{m+1};\b{u})=0, \quad \mbox{for $\big( \binom{l_{i}+1}{2}-t \big)_{+}$ vectors} \; \; \b{u}, \; \mbox{such that} \; |\b{u}| = 2.    \]
Hence it follows as before (cf. proof of (\ref{hea1})), that we require
\[ n > V(s_{3},s_{2},s_{1};p) \]
where
\[ s_{3} = \sum_{i=0}^{4} \bigg( \binom{l_{i}+1}{2}-t \bigg)_{+}\;, \qquad s_{2} = \sum_{i=0}^{4} \binom{l_{i}+2}{3}, \qquad s_{1} = \sum_{i=0}^{4} \binom{l_{i}+3}{4}.  \]
with the understanding that $\binom{a}{b}=0$ if $a<b$.
Moreover we can easily ensure that $\b{e}_{m+1}$ satisfies (i), by multiplying it by an appropriate power of $p$. So we have just shown that we can construct a set $S$ of arbitrary size $m$, in which the conditions (i) and (iii) hold. We should now remark that if $m > 5(s-1)$ then (ii) automatically holds. It is clear that the scenario in which $V(s_{1},s_{2},s_{3};p)$ gives the largest bound is when, $l_{i}=s-1$ for each $0 \le i \le 4$. Hence we are guaranteed to have a $(s,t)$-admissible set $S$ provided
\[ n > V(r_{3},r_{2},r_{1};p) \]
where
\[ r_{3} = 5\bigg( \binom{s}{2}-t \bigg)_{+}\;, \quad r_{2} = 5\binom{s+1}{3}, \quad r_{1} = 5\binom{s+2}{4}, \quad \mbox{as required}.  \]
\end{proof}

We are now set up to prove Theorems \ref{main1} and \ref{main2}. Note the following results, which are due to MAGMA \cite{Magma} calculations. 

\begin{result} \label{magma1}
Suppose we are working over $\F_{p}$, for $p \in \{2,3\}$. Consider the form
\[ f(t_{1},t_{2},t_{3}) = a_{1}t_{1}^{5} + b_{12}t_{1}t_{2}^{4} + a_{2}t_{2}^{5} + (b_{13}t_{1} + b_{23}t_{2})t_{3}^{4} + a_{3}t_{3}^{5} \]
where $a_{1}a_{2}a_{3} \ne 0$. Then $f$ has at least one non-singular zero over $\F_{p}$.
\end{result}

Suppose $n > V(15,20,25;p)$, then by Lemma \ref{kad} we have a $(3,0)$-admissible set $S$ and hence an associated subset $S_{r}= \{ \b{e}_{1}, \b{e}_{2}, \b{e}_{3} \}$ as in (ii). So by Lemma \ref{coeff}, the form
\begin{eqnarray*}
g(t_{1},t_{2},t_{3}) &:=& \theta_{1}(p^{-r}F(t_{1} \b{e}_{1} + t_{2} \b{e}_{2} + t_{3} \b{e}_{3})) \\
&=& \sum_{i}\bar{a}_{i}t_{i}^{5}+\sum_{i < j}\bar{b}_{ij}t_{i}t_{j}^{4}
\end{eqnarray*}
is such that $\bar{a}_{i},\bar{b}_{ij} \in \F_{p}$ and $\bar{a}_{i} \ne 0$. Hence by Result \ref{magma1}, $g(t_{1},t_{2},t_{3})$ has a non-singular zero over $\F_{p}$ for $p \in \{2,3\}$. Consequently by Hensel's Lemma $F$ has a non-trivial zero over $\Z_{p}$; a contradiction to our original assumption. So we have just proved that for $p \in \{2,3\}$

\begin{equation} \label{res1}
v_{5}(p) \le V(15,20,25;p).
\end{equation}

\begin{result} \label{magma1*}
Suppose we are working over $\F_{p}$, for $p \in \{7,13\}$. Consider the form
\[ f(t_{1},t_{2},t_{3}) = a_{1}t_{1}^{5} + b_{12}t_{1}t_{2}^{4} + a_{2}t_{2}^{5} + c_{123}t_{1}t_{2}t_{3}^{3} + (b_{13}t_{1} + b_{23}t_{2})t_{3}^{4} + a_{3}t_{3}^{5}  \]
where $a_{1}a_{2}a_{3} \ne 0$. Then $f$ has at least one non-singular zero over $\F_{p}$.
\end{result}

Suppose $n > V(10,20,25;p)$, then by Lemma \ref{kad} we have a $(3,1)$-admissible set $S$ and hence an associated subset $S_{r}= \{ \b{e}_{1}, \b{e}_{2}, \b{e}_{3} \}$ as in (ii). So by Lemma \ref{coeff}, the form
\begin{eqnarray*}
g(t_{1},t_{2},t_{3}) &:=& \theta_{1}(p^{-r}F(t_{1} \b{e}_{1} + t_{2} \b{e}_{2} + t_{3} \b{e}_{3})) \\
&=& \sum_{i}\bar{a}_{i}t_{i}^{5}+\sum_{i < j}\bar{b}_{ij}t_{i}t_{j}^{4} + \bar{c}_{123}t_{1}t_{2}t_{3}^{3}
\end{eqnarray*}
is such that $\bar{a}_{i},\bar{b}_{ij},\bar{c}_{123} \in \F_{p}$ and $\bar{a}_{i} \ne 0$. Hence by Result \ref{magma1*}, $g(t_{1},t_{2},t_{3})$ has a non-singular zero over $\F_{p}$ for $p \in \{7,13\}$. Consequently by Hensel's Lemma $F$ has a non-trivial zero over $\Z_{p}$; a contradiction to our original assumption. So we have just proved that for $p \in \{7,13\}$

\begin{equation} \label{res1*}
v_{5}(p) \le V(10,20,25;p).
\end{equation}

\begin{result} \label{magma2}
Suppose we are working over $\F_{11}$. Consider the form
\[ f(t_{1},t_{2},t_{3},t_{4}) = \sum_{i}a_{i}t_{i}^{5}+\sum_{i < j}b_{ij}t_{i}t_{j}^{4}
 + (c_{12}t_{1}t_{2}+c_{13}t_{1}t_{3}+c_{23}t_{2}t_{3})t_{4}^{3}  \]
where $a_{1}a_{2}a_{3}a_{4} \ne 0$. Then $f$ has at least one non-singular zero over $\F_{11}$.
\end{result}

So by exactly the same argument as above, if we have a $(4,3)$-admissible set we can deduce from Result \ref{magma2} that

\begin{equation} \label{res2}
v_{5}(11) \le V(15,50,75;11). 
\end{equation}

\begin{result} \label{magma3}
Suppose we are are working over $\Z_{5}$. Consider the form
\begin{eqnarray*}
f(t_{1},\ldots,t_{6}) &=& \sum_{i}a_{i}t_{i}^{5}+\sum_{i < j}b_{ij}t_{i}t_{j}^{4} + (c_{125}t_{1}t_{2}+c_{135}t_{1}t_{3}+c_{145}t_{1}t_{4}+c_{235}t_{2}t_{3})t_{5}^{3} \\ &+& \sum_{i<j<6}c_{ij6}t_{i}t_{j}t_{6}^{3}
\end{eqnarray*}
where $a_{1}\cdots a_{6} \not\equiv 0$ (mod $5$). Then there exists some $\b{t} \in \Z_{5}^{6}$ and $1 \le i \le 6$ such that
\[ f(\b{t}) \equiv 0 \; \mbox{(mod $25$)}, \quad \frac{\partial f}{\partial t_{i}}(\b{t}) \equiv 0 \; \mbox{(mod $5$)}, \quad \frac{\partial f}{\partial t_{i}}(\b{t}) \not\equiv 0 \; \mbox{(mod $25$)} \]
\[ and \qquad \frac{\partial^{2} f}{\partial t_{i}^{2}}(\b{t}) \equiv 0 \; \mbox{(mod $5$)}. \]
In particular we have $\b{t} \not\equiv \b{0}$ (mod $5$). 
\end{result}

We need to be careful in dealing with this case. To be clear we are constructing our set $\{ \b{e}_{1},\ldots,\b{e}_{6} \}$ in a way that the associated form takes the shape given in Result \ref{magma3}. In finding the vector $\b{e}_{5}$ we need an $(5,4)$-admissible set. Therefore by Lemma \ref{kad} we require $n > V(30,100,175)$. Moreover in finding the vector $\b{e}_{6}$, it is not difficult to show that we require $n > V(25,175,350)$. Using our reduction analysis on systems of cubic and quadratic forms, it transpires that on taking $n > V(30,100,175)$ we get the largest lower bound of $n$. Therefore to get our quintic form in the shape as in Result \ref{magma3} we require $n > V(30,100,175)$. The difference in this case opposed to the cases when $p \ne 5$, is that we look at the form 
\begin{eqnarray*}
g(t_{1},\ldots,t_{6}) &:=& \theta_{2}(p^{-r}F(t_{1}\b{e}_{1}+\ldots+t_{6}\b{e}_{6})) \\
&=& \sum_{i}\bar{a}_{i}t_{i}^{5}+\sum_{i < j}\bar{b}_{ij}t_{i}t_{j}^{4} + \sum_{i<j<k}\bar{c}_{ijk}t_{i}t_{j}t_{k}^{3}
\end{eqnarray*}
Where we have $\bar{a}_{i},\bar{b}_{ij}, \bar{c}_{ijk} \in \Z/p^{2}\Z$ with $5\bar{a}_{i} \ne 0$. So for $p=5$ we deduce that $F$ has a non trivial zero over $\Z_{p}$ (a contradiction) by applying Result \ref{magma3} along with the following Lemma.

\begin{lemma} \label{A3}
Let $F(x) \in \Z_{p}[x]$ be a polynomial of any degree. Suppose there exists some $a \in \Z_{p}$, such that
\begin{equation} \label{henhyp2}
p^{2} \mid F(a), \quad p \parallel F'(a) \quad and \quad p \mid F''(a)
\end{equation}
then there exists an $\alpha \in \Z_{p}$ such that, $F(\alpha)=0$. Moreover, $\alpha \equiv a$ (mod $p$).
\end{lemma}
In order not to interrupt our discussion we shall postpone the proof of this Lemma until the end of the present section. We have just proved that 

\begin{equation} \label{res3}
v_{5}(5) \le V(30,100,175;5).
\end{equation}

We now have all the information we need to go ahead and prove Theorems \ref{main1} and \ref{main2}. We shall use (\ref{cubicbd}) to deal with the primes $p \in \{2,3,5,11\}$ and for $p \in \{7,13\}$ we shall use Theorem \ref{newresult}. By Theorem \ref{newresult} we can prove that for $p \ne 3$ we have
\begin{equation} \label{newcubicbd}
V(a,b,c;p) \le V(\alpha,0;p) + \beta
\end{equation}
where
\[ \alpha = b + 3a(a-1), \]
\[ \beta = c + 3a(2b+3a) + \frac{1}{2}a(a-1)(24a-21). \]
\underline{$p=2$}: By (\ref{impcubicbd}) we have
\[ V(15,20,25;p) \le V(380,0;p)+11725. \]
By using Lemma \ref{sysquad} we get
\[ V(380,0;p) \le 288784. \]
Hence we obtain
\[ v_{5}(2) \le 300509 \]
as required. \\ \\
\underline{$p=3$}: By (\ref{impcubicbd}) we have
\[ V(15,20,25;p) \le V(500,0;p) +20345. \]
By using Lemma \ref{sysquad} we get
\[ V(500,0;p) \le 499984. \] 
Hence we obtain 
\[ v_{5}(3) \le 520329 \] 
as required. \\ \\
\underline{$p=5$}: By (\ref{impcubicbd}) we have
\[ V(30,100,175;p) \le V(1495,0;p) + 92875. \]
By using Lemma \ref{sysquad} we get
\[ V(1570,0;p) \le 4470036. \]
Hence we obtain
\[ v_{5}(5) \le 4562911 \]
as required. \\ \\
\underline{$p=7$}: By (\ref{newcubicbd}) we have
\[ V(10,20,25;p) \le V(290,0;p) + 11980. \]
By using Lemma \ref{sysquad} we get
\[ V(290,0;p) \le 168184. \]
Hence we obtain
\[ v_{5}(7) \le 180164\]
as required. \\ \\
\underline{$p=11$}: By (\ref{impcubicbd}) we have
\[ V(15,50,75;p) \le V(410,0;p) + 13125. \] 
By using Lemma \ref{L5}
\[ V(410,0;p) \le 335372. \]
Hence we obtain
\[ v_{5}(11) \le 348497 \]
as required. \\ \\
\underline{$p=13$}: By (\ref{newcubicbd}) we have 
\[ V(10,20,25;p) \le V(290,0;p) + 11980. \]
By using Lemma \ref{L5}
\[ V(290,0;p) \le 167612. \]  
Hence we obtain
\[ v_{5}(13) \le 179592 \]
as required.

It remains to prove Lemma \ref{A3}. By (\ref{henhyp2}) we have, $F(a) = p^{2} \bar{F}(a)$, $F'(a) = p \bar{F'}(a)$  where $p \nmid \bar{F'}(a)$ and $F''(a) = p \bar{F''}(a)$, for some polynomial $\bar{F}(x)$. Now consider
\begin{eqnarray*}
F(a+pa') &\equiv& F(a) + pa'F'(a) + p^{2}a'^{2}\frac{F''(a)}{2} \qquad \mbox{(mod $p^{3}$)} \\
&\equiv& p^{2}(\bar{F}(a) + a' \bar{F'}(a)) \qquad \qquad \qquad \; \; \; \mbox{(mod $p^{3}$)}
\end{eqnarray*}
Hence to find an $a'$ such that $F(a+pa') \equiv 0$ (mod $p^{3}$), we require
\[ \bar{F}(a) + a'\bar{F'}(a) \equiv 0 \; \; \mbox{(mod $p$)} \]
However since $p \nmid \bar{F'}(a)$, this is always possible. 
It follows by repeating this process with $\alpha_{r}=a+\ldots+p^{r}a^{(r)}$, we can always find some $a^{(r)}$ such that $F(\alpha_{r}) \equiv 0$ (mod $p^{r+2}$). We complete the proof by taking $\alpha = \lim_{r \to \infty}\alpha_{r}$.

\section{Proof of Theorem \ref{newresult}}

We can develop Heath-Brown's \cite{Hea08} hybrid approach for a single cubic form to deal with a system of cubic and quadratic forms. Many of the key ideas required to prove Theorem \ref{newresult} have already been introduced in the previous section. Suppose we have a system $\b{F} = (f^{(i,j)})$ where $f^{(i,j)}(x_{1},\ldots,x_{n}) \in \Q_{p}[x_{1},\ldots,x_{n}]$ is a form of degree $i$ for $0 \le j \le r_{i}$ and $1 \le i \le 3$. We will of course assume throughout that $r_{3}>0$, since otherwise we are reduced to a question concerning a system of quadratic forms. Our approach will be to apply the hybrid reduction to one of the cubic forms in the system, whilst simultaneously requiring that the vectors we obtain in our ``admissible'' set are such that the rest of the system vanishes in the span of this set. Let $\b{F}=\{F\} \cup \b{G}$ where $F = f^{(3,1)}$. Suppose that for each $\x \ne \b{0}$ such that $\b{G}(\x)=0$ we have $F(\x) \ne 0$. From now on let 
\[\tilde{r}_i=\left\{\begin{array}{cc} r_d-1,& \;\; i=3\\ r_i,&\;\;
    i<d \end{array}\right.\]
For $\x \ne \b{0}$ such that $\b{G}(\x)=0$ we shall say that $\x$ has ``level $r$'' if $0 \le r \le 2$ and $v(F(\x)) \equiv r$ (mod $3$). Clearly the level of $\x$ is a well defined property since for any such $\x$ we have assumed that $F(\x) \ne 0$. For any set 
\[ S = \{ \b{e}_{1}, \ldots, \b{e}_{m} \} \subset \Q_{p}^{n}-\{ \b{0} \}  \]
We shall say that $S$ is ``$\b{F}$-admissible'' provided that the following conditions hold

\begin{enumerate}
\item[(i)]
For each vector $\b{e}_{i} \in S$, we have $0 \le v(F(\b{e}_{i})) \le 2$
\item[(ii)]
There is some subset $S_{r} \subseteq S$ of cardinality 3, whose vectors all have level $r$
\item[(iii)]
If $\b{e}_{i_{1}},\ldots,\b{e}_{i_{l}} \in S$ all have the same level, with $i_{1} < i_{2} < \ldots < i_{l}$ for $l \le 3$ then, by writing $\b{e}_{i_{r}} = \b{e}_{r}'$ we have
\[ F(t_{1}\b{e}_{1}' + \ldots + t_{l}\b{e}_{l}') = \sum_{1 \le i \le l} a_{i}t_{i}^{3} + \sum_{1 \le i < j \le l} b_{ij}t_{i}t_{j}^{2} \]
for some $a_{i}, b_{ij} \in \Q_{p}$ and
\[ f^{(i,j)}(t_{1}\b{e}_{1}' + \ldots + t_{l}\b{e}_{l}') = 0 \]
for any $t_{1},\ldots,t_{l} \in \Q_{p}$, where $0 \le j \le \tilde{r}_{i}$ and $1 \le i \le 3$
\end{enumerate}
We shall now show that we can find an $\b{F}$-admissible set provided $n$ is sufficiently large

\begin{lemma} \label{F-ad}
With the same hypothesis as above, we can always find an $\b{F}$-admissible set provided 
\begin{equation} \label{cond1}
n > V(r_{3}-1,r_{2}+6(r_{3}-1),r_{1}+6r_{2}+9r_{3}). 
\end{equation}
\end{lemma}

\begin{proof}
As in the proof of Lemma \ref{kad}, we shall inductively construct an $\b{F}$-admissible set. We can choose $\b{e}_{1} \in \Q_{p}^{n} - \{\b{0}\}$, so that (iii) holds, since
\[ n > V(r_{3}-1,r_{2}+6(r_{3}-1),r_{1}+6r_{2}+9r_{3}) > V(r_{3}-1,r_{2},r_{1}). \]
Moreover we can ensure that (i) holds by multiplying $\b{e}_{1}$ by an appropriate power of $p$. Next suppose that (i) and (iii) hold for the set $S= \{ \b{e}_{1},\ldots,\b{e}_{m} \}$, where $m \ge 1$. We need to find a vector $\b{e}_{m+1}$ such that (i) and (iii) hold for $S \cup \{\b{e}_{m+1}\}$. Suppose that in $S$ there are $l_{r}$ vectors of level $r$, where $0 \le l_{r} < 3$ and $0 \le r \le 2$ (note that if $l_{r} \ge 3$ then $S$ is already $\b{F}$-admissible since (ii) holds). Moreover let 
\[ L_{r} = \{ \b{e}_{1}^{(r)}, \ldots, \b{e}_{l_{r}}^{(r)} \} \subseteq S \]
be the subset of $S$ which contains all the vectors of level $r$. Clearly $S = L_{0} \cup L_{1} \cup L_{2}$. Now consider

\[ F(t_{1}\b{e}_{1}^{(r)}+\ldots+t_{l_{r}}\b{e}_{l_{r}}^{(r)}+t\b{e}_{m+1}) =  \sum_{|\b{u}| \le 3} \b{t}^{\b{u}}t^{3-|\b{u}|}F_{r}(\b{e}_{m+1};\b{u}) \]
and

\[ f^{(i,j)}(t_{1}\b{e}_{1}^{(r)}+\ldots+t_{l_{r}}\b{e}_{l_{r}}^{(r)}+t\b{e}_{m+1}) = \sum_{|\b{u}| \le i} \b{t}^{\b{u}}t^{i-|\b{u}|}f_{r}^{(i,j)}(\b{e}_{m+1};\b{u})  \]
where 

\[ \b{t}^{\b{u}} = \prod_{j=1}^{l_{i}} t_{j}^{u_{j}}. \]
Then
 \[F_{r}(\b{e}_{m+1};\b{u}) \; \mbox{is a form of degree $3- |\b{u}|$ in $\b{e}_{m+1}$}, \]
and
 \[f^{(i,j)}_{r}(\b{e}_{m+1};\b{u}) \; \mbox{is a form of degree $i- |\b{u}|$ in $\b{e}_{m+1}$}. \]
\\ 
Since \emph{a priori} we don't know what the level of $\b{e}_{m+1}$ will be, it follows that for $S \cup \{\b{e}_{m+1}\}$ to satisfy (iii), we require 

\[ F_{r}(\b{e}_{m+1};\b{u})=0, \qquad \mbox{for all} \; \;  |\b{u}| = 1 \quad \mbox{and} \quad 0 \le r \le 2 \]
and
\[ f^{(i,j)}_{r}(\b{e}_{m+1};\b{u})=0, \qquad \mbox{for all $\b{u}$,} \qquad 1 \le j \le \tilde{r}_{i}, \; \; 1 \le i \le 3 \; \; \mbox{and} \; \; 0 \le r \le 2. \]
Hence it follows as before, that we require

\[ n > V(s_{3},s_{2},s_{1};p) \]
where
\[ s_{3} = \tilde{r}_{3}, \qquad s_{2} = \tilde{r}_{2} + \tilde{r}_{3}\sum_{i=0}^{2}l_{i}, \qquad s_{1} = \tilde{r}_{1} + \tilde{r}_{2}\sum_{i=0}^{2}l_{i} + \tilde{r}_{3}\sum_{i=0}^{2}\binom{l_{i}+1}{2} +9. \]
with the understanding that $\binom{a}{b}=0$ if $a<b$.
Moreover we can easily ensure that $\b{e}_{m+1}$ satisfies (i), by multiplying it by an appropriate power of $p$. So we have just shown that we can construct a set $S$ of arbitrary size $m$, in which the conditions (i) and (iii) hold. We should now remark that if $m > 6$ then (ii) automatically holds. It is clear that the scenario in which $V(s_{1},s_{2},s_{3};p)$ gives the largest bound is whenever, $l_{i}=2$ for each $0 \le i \le 4$. Hence we are guaranteed to have a $\b{F}$-admissible set $S$ provided
\[ n > V(r_{3}-1,r_{2}+6(r_{3}-1),r_{1}+6r_{2}+9r_{3};p) \]
as required.
\end{proof}

We proceed by noting a key result as found in Heath-Brown's paper \cite[Lemma 6]{Hea08}.

\begin{lemma} \label{cubic-ad}
Suppose $p \ne 3$ and consider
\[ f(x,y,z) = ax^{3} + bxy^{2} + cy^{3} + (dx+ey)z^{2} + fz^{3} \in \F_{p}[x,y,z] \]
where $acf \ne 0$. Then $f$ has at least one non-singular zero over $\F_{p}$.
\end{lemma}

We are now in a position to complete the proof of Theorem \ref{newresult}. By Lemma \ref{F-ad}, we can find an $\b{F}$-admissible set provided that (\ref{cond1}) holds. So without loss generality there exists some set
\[ S_{r} = \{ \b{e}_{1},\b{e}_{2},\b{e}_{3} \} \]
which satisfies (ii) and (iii). By Lemma \ref{linind} $S$ is a linearly independent set, since the form
\[ F( t_{1}\b{e}_{1}+\ldots+t_{j}\b{e}_{j}) \]
contains no terms of degree one in $t_{j}$ for $j \le 3$. Moreover by Lemma \ref{coeff}, we have that the form
\begin{equation} \label{F}
p^{-r}F( t_{1}\b{e}_{1}+t_{2}\b{e}_{2}+t_{3}\b{e}_{3}) = \sum_{1 \le i \le 3}a_{i}t_{i}^{3} + \sum_{1\le i < j \le 3}b_{ij}t_{i}t_{j}^{2}
\end{equation}
is such that $a_{i}, b_{ij} \in \Z_{p}$ and $v(a_{i})=0$. So by looking at the form (\ref{F}) modulo $p$, we can apply Lemma \ref{cubic-ad} to deduce that $F$ has a zero $\x \ne 0$ such that $\b{G}(\x)=0$. This is a contradiction, hence we complete the proof of Theorem \ref{newresult}.

\newpage

\bibliography{bibo}
\bibliographystyle{plain}

\end{document}